\documentclass[mathpazo]{cicp}
\usepackage{epsf,amsmath,amsfonts,amsthm,graphicx,color,enumitem}
\numberwithin{equation}{section}

\begin{document}

\theoremstyle{plain}
\newtheorem{Lemma}{Lemma}[section]
\newtheorem{Prop}[Lemma]{Proposition}
\newtheorem{Thm}[Lemma]{Theorem}
\newtheorem{Cor}[Lemma]{Corollary}

\theoremstyle{definition}
\newtheorem{Def}[Lemma]{Definition}
\newtheorem{Rk}[Lemma]{Remark}
\newtheorem{Example}[Lemma]{Example}
\newtheorem{Exercise}[Lemma]{Exercise}

\newcommand{\Natural}{\mbox{${\bf N}$}}
\newcommand{\Integer}{\mbox{${\bf Z}$}}
\newcommand{\Real}{\mbox{${\bf R}$}}
\newcommand{\Complex}{\mbox{${\bf C}$}}

\newcommand{\Eps}{\varepsilon}
\newcommand{\Alpha}{A}
\newcommand{\Beta}{B}
\newcommand{\Epsilon}{E}
\newcommand{\Zeta}{Z}
\newcommand{\Eta}{H}
\newcommand{\Iota}{I}
\newcommand{\Kappa}{K}
\newcommand{\Mu}{M}
\newcommand{\Nu}{N}
\newcommand{\omicron}{o}
\newcommand{\Omicron}{O}
\newcommand{\Rho}{R}
\newcommand{\Tau}{T}
\newcommand{\Chi}{X}

\newcommand{\Sfrac}[2]{\mbox{\small$\frac{#1}{#2}$\normalsize}}
\newcommand{\Half}{\Sfrac{1}{2}}
\newcommand{\Vect}[1]{{\bf #1}}
\newcommand{\Grad}[1]{\nabla #1}
\newcommand{\Gradp}[1]{\nabla' #1}
\newcommand{\Gradx}[1]{\nabla_x #1}
\newcommand{\Gradxp}[1]{\nabla_{x'} #1}
\newcommand{\GradAlphap}[1]{\nabla_{\alpha}' #1}
\newcommand{\Div}[1]{\text{div}\left[#1\right]}
\newcommand{\Divp}[1]{\text{div}'\left[#1\right]}
\newcommand{\Divx}[1]{\text{div}_x \left[#1\right]}
\newcommand{\Divxp}[1]{\text{div}_{x'} \left[#1\right]}
\newcommand{\Curl}[1]{\text{curl}[#1]}
\newcommand{\Laplacian}[1]{\Delta #1}
\newcommand{\Laplacianp}[1]{\Delta' #1}
\newcommand{\Laplacianx}[1]{\Delta_x #1}
\newcommand{\Laplacianxp}[1]{\Delta_{x'} #1}
\newcommand{\Biharmonic}[1]{\Delta^2 #1}
\newcommand{\FT}[1]{{\cal F} \left\{ #1 \right\}}
\newcommand{\FTI}[1]{{\cal F}^{-1} \left\{ #1 \right\}}
\newcommand{\Variation}[2]{\delta_{#2} #1}

\newcommand{\Norm}[2]{\left\|#1\right\|_{#2}}
\newcommand{\LeftNorm}[1]{\left\|#1\right.}
\newcommand{\RightNorm}[2]{\left.#1\right\|_{#2}}
\newcommand{\SupNorm}[1]{\left|#1\right|_{L^{\infty}}}
\newcommand{\HolderNorm}[2]{\left|#1\right|_{C^{#2}}}
\newcommand{\SobNorm}[2]{\left\|#1\right\|_{H^{#2}}}

\newcommand{\InnerProd}[2]{\left\langle#1,#2\right\rangle}
\newcommand{\DotProd}[2]{\left\langle#1,#2\right\rangle}
\newcommand{\Abs}[1]{\left|#1\right|}
\newcommand{\Mod}[1]{\left|#1\right|}
\newcommand{\Angle}[1]{\langle #1 \rangle}
\newcommand{\RealPart}[1]{\text{Re\{}#1\text{\}}}
\newcommand{\ImagPart}[1]{\text{Im\{}#1\text{\}}}
\newcommand{\Null}[1]{\mbox{${\cal N}$}(#1)}
\newcommand{\Ran}[1]{\text{ran}(#1)}
\newcommand{\Ker}[1]{\text{ker}(#1)}
\newcommand{\Dim}[1]{\text{dim}(#1)}
\newcommand{\Rank}[1]{\text{rank}(#1)}
\newcommand{\Det}[1]{\mbox{det} #1}
\newcommand{\Span}[1]{\text{span}(#1)}
\newcommand{\sgn}{\text{sgn}}

\newcommand{\sech}{\mbox{$\mathrm{sech}$}}
\newcommand{\csch}{\mbox{$\mathrm{csch}$}}

\newcommand{\Intersect}{\cap}
\newcommand{\Union}{\cup}

\newcommand{\dftl}[1]{\; d#1}
\newcommand{\dV}{\dftl{V}}
\newcommand{\dS}{\dftl{S}}
\newcommand{\dx}{\dftl{x}}
\newcommand{\dy}{\dftl{y}}
\newcommand{\dz}{\dftl{z}}
\newcommand{\ds}{\dftl{s}}
\newcommand{\dt}{\dftl{t}}
\newcommand{\du}{\dftl{u}}
\newcommand{\dsigma}{\dftl{\sigma}}

\newcommand{\BigOh}[1]{\mathcal{O}(#1)}
\newcommand{\LittleOh}[1]{\mathcal{o}(#1)}

\newcommand{\cA}{\mathcal{A}}
\newcommand{\cB}{\mathcal{B}}
\newcommand{\cC}{\mathcal{C}}
\newcommand{\cD}{\mathcal{D}}
\newcommand{\cE}{\mathcal{E}}
\newcommand{\cF}{\mathcal{F}}
\newcommand{\cG}{\mathcal{G}}
\newcommand{\cH}{\mathcal{H}}
\newcommand{\cI}{\mathcal{I}}
\newcommand{\cJ}{\mathcal{J}}
\newcommand{\cK}{\mathcal{K}}
\newcommand{\cL}{\mathcal{L}}
\newcommand{\cM}{\mathcal{M}}
\newcommand{\cN}{\mathcal{N}}
\newcommand{\cO}{\mathcal{O}}
\newcommand{\cP}{\mathcal{P}}
\newcommand{\cQ}{\mathcal{Q}}
\newcommand{\cR}{\mathcal{R}}
\newcommand{\cS}{\mathcal{S}}
\newcommand{\cT}{\mathcal{T}}
\newcommand{\cU}{\mathcal{U}}
\newcommand{\cV}{\mathcal{V}}
\newcommand{\cW}{\mathcal{W}}
\newcommand{\cX}{\mathcal{X}}
\newcommand{\cY}{\mathcal{Y}}
\newcommand{\cZ}{\mathcal{Z}}

\newcommand{\px}{\partial_x}
\newcommand{\py}{\partial_y}
\newcommand{\pz}{\partial_z}
\newcommand{\pt}{\partial_t}

\newcommand{\sumn}{\sum_{n=0}^{\infty}}
\newcommand{\sumno}{\sum_{n=1}^{\infty}}
\newcommand{\sumk}{\sum_{k=-\infty}^{\infty}}
\newcommand{\sump}{\sum_{p=-\infty}^{\infty}}

\newcommand{\summ}{\sum_{m=0}^{\infty}}
\newcommand{\sumell}{\sum_{\ell=0}^{\infty}}

\newcommand{\be}{\begin{equation}}
\newcommand{\ee}{\end{equation}}
\newcommand{\bes}{\begin{equation*}}
\newcommand{\ees}{\end{equation*}}
\newcommand{\bse}{\begin{subequations}}
\newcommand{\ese}{\end{subequations}}

\newcommand{\Schrodinger}{Schr\"odinger}
\newcommand{\Holder}{H\"older}
\newcommand{\Calderon}{Calder\'{o}n}
\newcommand{\Pade}{Pad\'{e}}

\newcommand{\Question}[1]{\fbox{ {\bf Q: #1} }}
\newcommand{\Corrected}[1]
{\noindent \rule{\linewidth}{.75mm} \\ {\bf CORRECTED UP TO HERE (#1)} \\ \rule{\linewidth}{.75mm}}

\newcommand{\pN}{\partial_N}
\newcommand{\Vu}{V^u}
\newcommand{\Vl}{V^{\ell}}
\newcommand{\tVu}{\tilde{V}^u}
\newcommand{\tVl}{\tilde{V}^{\ell}}
\newcommand{\Huu}{H^{u u}}
\newcommand{\Hul}{H^{u \ell}}
\newcommand{\Hlu}{H^{\ell u}}
\newcommand{\Hll}{H^{\ell \ell}}
\newcommand{\bA}{\mathbf{A}}
\newcommand{\bM}{\mathbf{M}}
\newcommand{\bV}{\mathbf{V}}
\newcommand{\bR}{\mathbf{R}}
\newcommand{\bL}{\mathbf{L}}
\newcommand{\bD}{\mathbf{D}}
\newcommand{\bU}{\mathbf{U}}

\newcommand{\vmmo}{v^{(m-1)}}
\newcommand{\vm}{v^{(m)}}
\newcommand{\xim}{\xi^{(m)}}
\newcommand{\num}{\nu^{(m)}}
\newcommand{\taummo}{\tau_{m-1}}
\newcommand{\taum}{\tau_m}
\newcommand{\sigmam}{\sigma^{(m)}}
\newcommand{\sigmamhat}{\hat{\sigma}^{(m)}}
\newcommand{\epsilonm}{\epsilon_m}
\newcommand{\Nm}{N^{(m)}}
\newcommand{\pNm}{\partial_{N^{(m)}}}
\newcommand{\pNmpo}{\partial_{N^{(m+1)}}}
\newcommand{\am}{a^{(m)}}
\newcommand{\ampo}{a^{(m+1)}}
\newcommand{\gm}{g^{(m)}}
\newcommand{\gmpo}{g^{(m+1)}}
\newcommand{\Am}{A^{(m)}}
\newcommand{\Bm}{B^{(m)}}

\newcommand{\ku}{k_u}
\newcommand{\riu}{n_u}
\newcommand{\gammau}{\gamma_u}

\newcommand{\Vml}{V^{(m),\ell}}
\newcommand{\Vmu}{V^{(m),u}}
\newcommand{\tVml}{\tilde{V}^{(m),\ell}}
\newcommand{\tVmu}{\tilde{V}^{(m),u}}
\newcommand{\chim}{\chi^{(m)}}
\newcommand{\psim}{\psi^{(m)}}

\newcommand{\Vzl}{V^{(0),\ell}}
\newcommand{\VMu}{V^{(M),u}}
\newcommand{\tVzl}{\tilde{V}^{(0),\ell}}
\newcommand{\tVMu}{\tilde{V}^{(M),u}}

\newcommand{\Ruu}{R^{u u}}
\newcommand{\Rul}{R^{u \ell}}
\newcommand{\Rlu}{R^{\ell u}}
\newcommand{\Rll}{R^{\ell \ell}}

\newcommand{\void}[1]{}

\newcommand{\RevOne}[1]{\textcolor{red}{#1}}
\newcommand{\RevTwo}[1]{\textcolor{blue}{#1}}

\newcommand{\hsigmaD}{\hat{\sigma}_{\text{Drude}}}
\newcommand{\DeltaTE}{\Delta^{\text{TE}}}
\newcommand{\DeltaTM}{\Delta^{\text{TM}}}

\newcommand{\Tianyu}[1]{\textcolor{green}{#1}}

\newcommand{\hsigmaBGK}{\hat{\sigma}_{\text{BGK}}}
\newcommand{\hsigmaloc}{\hat{\sigma}_{\text{loc}}}
\newcommand{\hsigmanloc}{\hat{\sigma}_{\text{nloc}}}
\newcommand{\Aloc}{A_{\text{loc}}}
\newcommand{\Anloc}{A_{\text{nloc}}}
\newcommand{\Bloc}{B_{\text{loc}}}
\newcommand{\Bnloc}{B_{\text{nloc}}}

%
%

\title[HOS Simulation of 2D Materials]
  {High--Order Spectral Simulation of Dispersive
  Two--Dimensional Materials}
\author[D.P.~Nicholls et.~al.]
  {David P.\ Nicholls\affil{1}\comma\corrauth
  and Tianyu Zhu\affil{1}}
\address{\affilnum{1}
  Department of Mathematics, Statistics, and Computer Science,
  University of Illinois at Chicago, Chicago, IL 60607}
\email{\texttt{davidn@uic.edu} (D.P.~Nicholls),
  \texttt{tzhu29@uic.edu} (T.~Zhu)}

\begin{abstract}
Over the past twenty years, the field of plasmonics has been
revolutionized with the isolation and utilization of 
two--dimensional materials, particularly graphene.
Consequently there is significant interest in rapid, robust, and
highly accurate computational schemes which can incorporate such
materials. Standard volumetric approaches can be contemplated, but these
require huge computational resources. Here we describe an
algorithm which addresses this issue for nonlocal models of
the electromagnetic response of graphene. Our methodology not only
approximates the graphene layer with a surface current, but also 
reformulates the
governing volumetric equations in terms of surface quantities using
Dirichlet--Neumann Operators. We have recently shown how these
surface equations can be
numerically simulated in an efficient, stable, and accurate fashion
using a High--Order Perturbation of Envelopes methodology. We extend
these results to the nonlocal model mentioned above, and using
an implementation of this algorithm, we study absorbance spectra of TM
polarized plane--waves scattered by a periodic grid of graphene
ribbons.
\end{abstract}

\ams{78A45, 65N35, 78B22, 35J05, 41A58}
\keywords{two--dimensional materials, graphene,
  non--local current models,
  electromagnetic scattering, high--order spectral methods.}

\maketitle

%
%

\section{Introduction}
\label{Sec:Intro}

Over the past twenty years, the field of plasmonics has been
revolutionized with the isolation and utilization of 
two--dimensional materials, particularly graphene. Graphene 
is a single layer of carbon atoms arranged in a honeycomb
lattice which has striking mechanical, chemical, and electronic
properties \cite{GeimNovoselov07,Geim09}. It was first isolated
in 2004 \cite{NGMJZDGF04} resulting in the awarding of the 2011
Nobel Prize to Geim \cite{Geim11} and Novoselov \cite{Novoselov11}.
At this point the literature on graphene is so vast that it is
impossible to describe even a fraction of it here, however, we
point the interested reader to the website maintained by 
\textit{Nature}
dedicated to the major developments in the field \cite{NatureGraphene}.
The authors have found the survey article of
Bludov, Ferriera, Peres, and Vasilevskiy \cite{BFPV13} and survey
book of Goncalves and Peres \cite{GoncalvesPeres16} to be particularly
helpful.

Among the many optical phenomena associated to graphene, the
collective charge oscillations known as plasmons
\cite{JablanBuljanSoljacic09,JablanSoljacicBuljan13} are distinguished.
Recently, the dispersive, nonlocal properties of these graphene plasmons
have generated interest in the engineering literature
\cite{FLowTP15,CSGDTA15,MorgadoSilveirinha20,ZhuWangGuo20,KwiecienBurdaRichter23}
and the object of this contribution is to initiate this study. In
particular, we describe a novel algorithm, inspired by our previous
work \cite{Nicholls19}, for simulating the scattering
returns by a periodic array of graphene strips which takes into
account the effects of nonlocality.
We point out that, in addition to the optical phenomena that we
have mentioned above, graphene has become indispensable
in applications
as diverse as energy storage \cite{OAWS21}, drug delivery and tumor
therapy \cite{SSWCXZS20}, biomedical devices \cite{SZLZ12}, strain
sensors \cite{LZWLJMC20}, and membranes \cite{NCBL21}.

Before beginning our description, we point out that among the
many techniques for numerically simulating structures featuring
graphene (or other two--dimensional materials), simply solving
the volumetric Maxwell equations in either the time domain
(e.g., the Time Domain Finite Difference method \cite{TafloveHagness00})
or frequency domain (e.g., the Finite Element Method \cite{Jin02})
are natural options \cite{GallinetButetMartin15}. Typically,
the graphene is modeled with an effective permittivity
supported in a \textit{thin} layer, or as a surface current with an
effective conductivity at the interface between two layers
\cite{HongNicholls21}. In either
case, commercial black--box Finite Element Method (FEM) software such
as COMSOL Multiphysics\texttrademark \cite{COMSOL} is typically utilized,
however, these simulations are quite
costly due to their low--order accuracy and volumetric character.

In our recent contributions \cite{Nicholls17,Nicholls19} we described
a method which overcomes these drawbacks by not only restating
the frequency domain governing equations in terms of \textit{interfacial}
unknowns, but also describing a highly accurate, efficient, and stable High--Order Spectral (HOS) algorithm
\cite{GottliebOrszag77,ShenTang06,ShenTangWang11}. A feature of our
algorithm is that, in order to close the system of equations, surface
integral operators must be introduced which connect interface traces 
of the scattered fields (Dirichlet data) to their surface normal
derivatives (Neumann data). Such Dirichlet--Neumann Operators (DNOs)
have been widely used and studied in the simulation of linear
wave scattering, e.g., interfacial formulations of scattering problems
\cite{Nicholls12,NichollsOhJohnsonReitich15,Nicholls17,Nicholls19}.

The object of our study is the plasmonic response that can be generated by
graphene and, as in many photonic devices, structural periodicity is
one path. This can be accomplished in several ways, and in one of our 
earlier papers \cite{Nicholls17} we focused upon graphene
deposited on a periodically corrugated grating. In this the height/slope
of the corrugation \textit{shape} was viewed as a perturbation parameter and
the resulting High--Order Perturbation of Shapes (HOPS) scheme sought
corrections to the trivially computed flat--interface, solid
graphene configuration. However, it is much more common to create a
structure with \textit{flat} interfaces upon which periodically
spaced ribbons of graphene are mounted. In the paper \cite{Nicholls19}
we modeled this configuration by multiplying the (constant) current function
by an \textit{envelope} function which transitions between one (where
the graphene is deposited) to zero (where graphene is absent). Our
numerical procedure viewed this envelope function as a perturbation
of the identity function, and we termed that scheme a High--Order
Perturbation of Envelopes (HOPE) algorithm. Our purpose in this
contribution is to extend these latter results to the case of a
nonlocal model for the response of a graphene layer. As we shall show,
as such a model introduces higher order derivatives to the governing
equations, this is a \textit{highly} non--trivial extension requiring
significant theoretical and algorithmic generalizations of those
found in \cite{Nicholls19}.

Using our HOPE method we not only rigorously demonstrate that the 
scattered fields depend analytically upon the envelope perturbation
parameter, but also show that the resulting numerical scheme 
is both robust and accurate,
and extremely rapid in its execution. As with the algorithm specified
in \cite{Nicholls19}, due to the flat interfaces of this geometry,
the relevant DNOs are reduced to simple Fourier
multipliers which can be easily computed in Fourier space. This is to
be contrasted to the case of corrugated interfaces from \cite{Nicholls17}
where a stable and accurate HOPS scheme for their computation is 
highly non--trivial to design and implement.

The rest of the paper is organized as follows: In Section~\ref{Sec:Govern}
we recall the governing equations of our model \cite{Nicholls17,Nicholls19} for
the response of a two--dimensional material mounted between two dielectrics.
In Section~\ref{Sec:Surf} we describe our surface formulation of these
equations, specializing to the patterned, flat--interface configuration 
in Section~\ref{Sec:PattFlat}. We prescribe our HOPE methodology
in Section~\ref{Sec:HOPE}.
We state and prove our analyticity results in Section~\ref{Sec:Anal}.
We conclude with numerical results in Section~\ref{Sec:NumRes},
with a discussion of implementation issues in Section~\ref{Sec:Impl}
and simulation of absorbance spectra in Section~\ref{Sec:AbsSpec}.

%
%

\section{Governing Equations}
\label{Sec:Govern}

Following \cite{Nicholls17}, the structure we consider is displayed
in Figure~\ref{Fig:Structure},
a doubly layered, $y$--invariant medium with periodic interface shaped
by $z = g(x)$, $g(x+d)=g(x)$. This interface separates two domains
filled with dielectrics of permittivities $\epsilon_u$ in 
$S^u := \{ z > g(x) \}$ and $\epsilon_w$ in $S^w := \{ z < g(x) \}$,
respectively. This is illuminated with time--harmonic (of dependence
$\exp(-i \omega t)$) plane--wave radiation of incidence
angle $\theta$, frequency $\omega$, and wavenumber 
$\ku = \sqrt{\epsilon_u} \omega/c_0$,
\bes
v^{\text{inc}} = e^{i (\alpha x - \gammau z)},
\quad
\alpha = \ku \sin(\theta),
\quad
\gammau = \ku \cos(\theta).
\ees
%
%
\begin{figure}[hbt]
\begin{center}
  \includegraphics[width=0.45\textwidth]{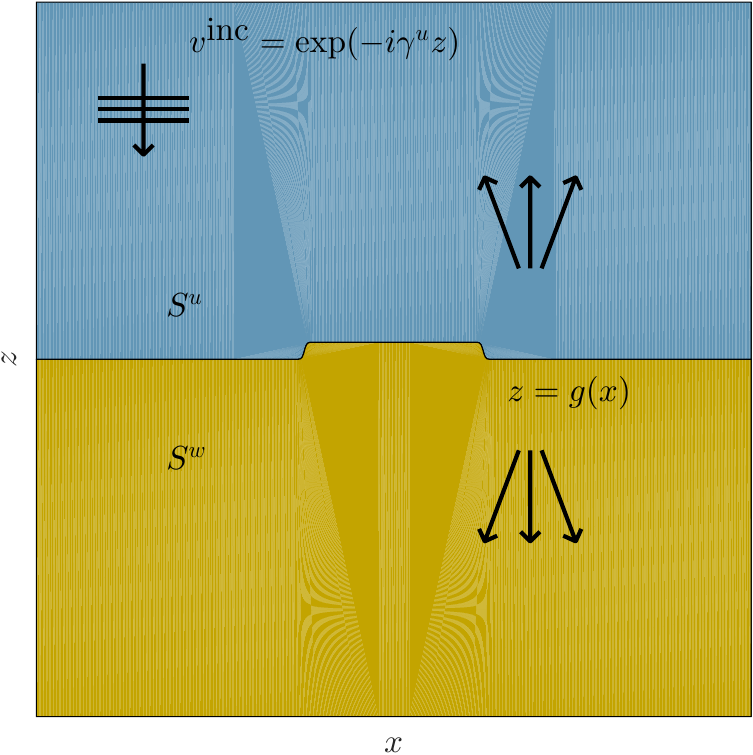}
  \caption{Plot of two--layer structure with periodic interface.}
  \label{Fig:Structure}
\end{center}
\end{figure}

As we detailed in \cite{Nicholls17}, if we choose as unknowns,
$\{ u(x,z), w(x,z) \}$, the laterally quasiperiodic transverse 
components of either the electric (Transverse Electric--TE)
or magnetic (Transverse Magnetic--TM) fields, then the governing
equations in this two--layer configuration are
\bse
\label{Eqn:Double}
\begin{align}
& u - w + A \tau_w \pN w = \xi, && z = g(x), \\
& \tau_u \pN u - \tau_w \pN w + B w = \tau_u \nu, && z = g(x),
\end{align}
\ese
where $\pN = N \cdot \Grad{}$, $N = (-\px g,1)^T$,
\bes
\tau_m = \begin{cases} 1, & \text{TE}, \\ 
  1/\epsilon_m, & \text{TM}, \end{cases}
\quad
A = \begin{cases} 0, & \text{TE}, \\
  \Abs{N} \hat{\sigma}/(i k_0), & \text{TM}, \end{cases}
\quad
B = \begin{cases} (i k_0) \hat{\sigma}/\Abs{N}, & \text{TE}, \\
  0, & \text{TM}, \end{cases}
\ees
for $m \in \{ u, w \}$, and
\bes
\xi(x) = -\left. v^{\text{inc}} \right|_{z=g(x)},
\quad
\nu(x) = -\left. \pN v^{\text{inc}} \right|_{z=g(x)}.
\ees
Of particular note is $\hat{\sigma} = \sigma/(\epsilon_0 c_0)$, the
dimensionless \textit{surface} current which models the effects of
the graphene (or other two--dimensional material) deposited at the
interface between the two layers.

%
%

\section{Surface Formulation}
\label{Sec:Surf}

Following \cite{Nicholls12,Nicholls17} we now reformulate the problem 
\eqref{Eqn:Double} in terms of surface integral operators, in this case 
Dirichlet--Neumann Operators (DNOs).  For this we define the Dirichlet
traces
\bes
U(x) := u(x,g(x)),
\quad
W(x) := w(x,g(x)),
\ees
and the outward pointing Neumann traces
\bes
\tilde{U}(x) := -(\pN u)(x,g(x)),
\quad
\tilde{W}(x) := (\pN w)(x,g(x)).
\ees
In terms of these \eqref{Eqn:Double} read
\bse
\label{Eqn:Double:DNO}
\begin{align}
& U - W + A \tau_w \tilde{W} = \xi, \\
& -\tau_u \tilde{U} - \tau_w \tilde{W} + B W = \tau_u \nu.
\end{align}
\ese
These specify two equations for four unknowns which would
be problematic except that $U$ and $\tilde{U}$
are connected, as are $W$ and $\tilde{W}$. We formalize this with the
following definitions \cite{Nicholls16b}.

\begin{Def}
Given the unique quasiperiodic upward propagating solution \cite{ArensHab} to
\be
\label{Eqn:Helm:u}
\Laplacian{u} + k_u^2 u = 0,
\quad
z > g(x),
\ee
subject to the Dirichlet condition, $u(x,g(x)) = U(x)$,
the Neumann data, $\tilde{U}(x)$, can be computed. The DNO
$G$ is defined by
\bes
G(g): U \rightarrow \tilde{U}.
\ees
\end{Def}
\begin{Def}
Given the unique quasiperiodic downward propagating solution \cite{ArensHab} to
\be
\label{Eqn:Helm:w}
\Laplacian{w} + k_w^2 w = 0,
\quad
z < g(x),
\ee
subject to the Dirichlet condition, $w(x,g(x)) = W(x)$,
the Neumann data, $\tilde{W}(x)$, can be computed. The DNO
$J$ is defined by
\bes
J(g): W \rightarrow \tilde{W}.
\ees
\end{Def}

Negating the second equation, \eqref{Eqn:Double:DNO} can now be written as
\be
\label{Eqn:DNO}
\begin{pmatrix} I & -I + A \tau_w J \\ 
  \tau_u G & \tau_w J - B \end{pmatrix}
\begin{pmatrix} U \\ W \end{pmatrix}
= \begin{pmatrix} \xi \\ -\tau_u \nu \end{pmatrix}.
\ee

%
%

\subsection{The Patterned, Flat--Interface Configuration}
\label{Sec:PattFlat}

The configurations of interest to engineers \cite{BFPV13,GoncalvesPeres16}
often feature \textit{flat} layer interfaces with \textit{patterned} 
graphene sandwiched in between. For this we use the modeling assumptions
\bes
g(x) \equiv 0,
\quad
\hat{\sigma} \approx \hsigmaBGK X(x;\delta),
\ees
where $\hsigmaBGK$ is a (dimensionless) 
Bhatnagar--Gross--Krook (BGK) model for the graphene
\cite{FLowTP15},
\bes
\hsigmaBGK = \hsigmaD \left\{ 1 - v_F^2 
  \left( \frac{3 f + 2 i/\tau}{4 f (f + i/\tau)^2} \right)
  \px^2 \right\}
  =: \hsigmaloc - \hsigmanloc \px^2,
\ees
where the local term comes from a Drude model
\cite{BFPV13,GoncalvesPeres16},
\bes
\hsigmaD = \frac{\sigma_0}{\epsilon_0 c_0} 
  \left( \frac{4 E_F}{\pi} \right) 
  \frac{1}{\hbar \tilde{\gamma} - i \hbar \omega}
  = \frac{2 E_F e^2}{\epsilon_0 c_0} 
  \left( \frac{1}{\Gamma - i h f} \right),
\ees
where $\sigma_0 = \pi e^2/(2 h)$ is the universal AC conductivity
of graphene \cite{GoncalvesPeres16}, $e>0$ is the elementary charge,
$h$ is Planck's constant, $\hbar = h/(2 \pi)$, $E_F$ is the (local) 
Fermi level position, and $\tilde{\gamma}$ is the relaxation rate.
($\Gamma = \hbar \tilde{\gamma}$ is another frequently used notation.)
Further, $v_F$ is the Fermi velocity, $f = \omega/(2 \pi)$ is the
ordinary frequency of the incident radiation, and $\tau$ 
is the carrier
lifetime. We will also have use for the following decomposition
of $A = \Aloc - \Anloc$ and $B = \Bloc - \Bnloc$,
\bes
\Aloc = \begin{cases} 0, & \text{TE}, \\
  \Abs{N} \hsigmaloc/(i k_0), & \text{TM}, \end{cases},
\quad
\Anloc = \begin{cases} 0, & \text{TE}, \\
  \Abs{N} (\hsigmanloc/(i k_0)) \px^2, & \text{TM}, \end{cases},
\ees
and
\bes
\Bloc = \begin{cases} (i k_0) \hsigmaloc/\Abs{N}, & \text{TE}, \\
  0, & \text{TM}, \end{cases},
\quad
\Bnloc = \begin{cases} (i k_0) (\hsigmanloc/\Abs{N}) \px^2, & \text{TE}, \\
  0, & \text{TM}, \end{cases}.
\ees

Also, $X(x;\delta)$ is a $d$--periodic (in $x$) \textit{envelope} 
function which we use to model the patterning. For this we permit
the envelope to be varied with a parameter $\delta$, e.g.,
\be
\label{Eqn:X}
X(x;\delta) = X_0 + \delta X_1(x),
\ee
where $X_0 \neq 0$ as explained below, and
\bes
X_1(x) = -X_0
  + \begin{cases} 
  \sqrt{ 1 - 4 \left( \frac{x-d/2}{w} \right)^2 }, & d/2-w/2 < x < d/2+w/2, \\
  0, & \text{else},
  \end{cases}
\ees
and $w$ is the ribbon width; see Figure~\ref{Fig:Envelope}. This profile
was specified in \cite{BFPV13} to model not only the patterning but also
edge effects.
%
%
%
%
%
%
\begin{figure}[hbt]
\begin{center}
  \includegraphics[width=0.45\textwidth]{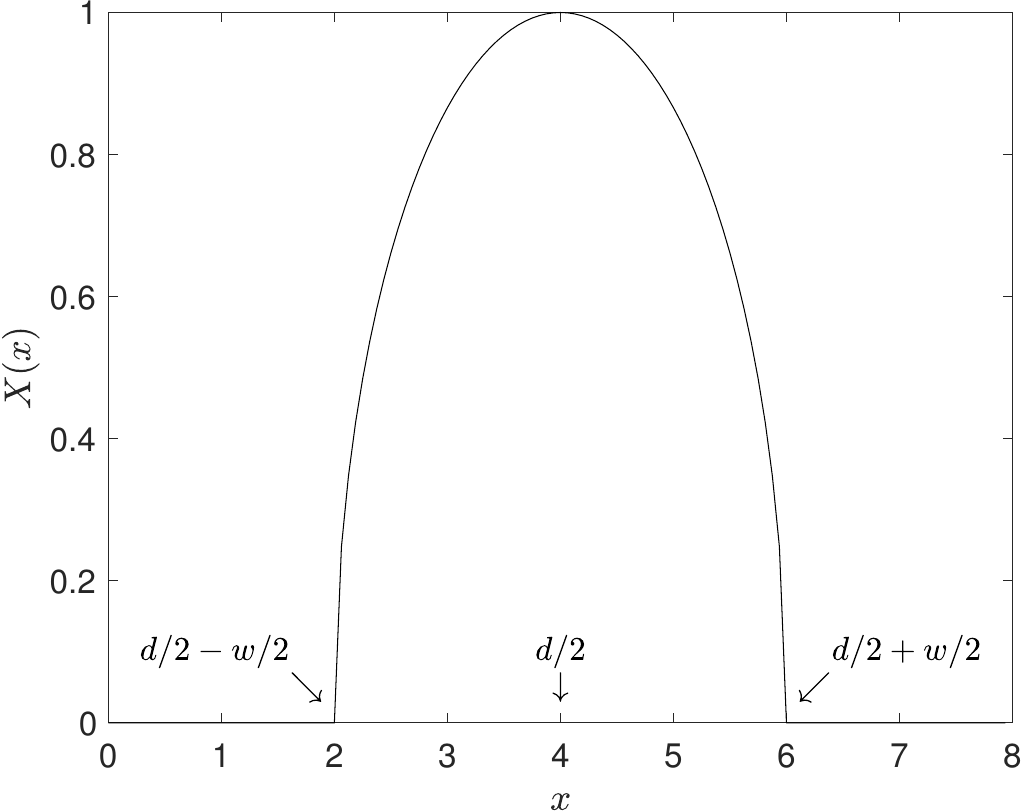}
  \caption{Plot of the current envelope function, 
    $X(x) = X_0 + X_1(x)$.}
  \label{Fig:Envelope}
\end{center}
\end{figure}

With these assumptions, and denoting $G_0 = G(0)$ and $J_0 = J(0)$,
we consider the modification of \eqref{Eqn:DNO},
\be
\label{Eqn:DNO:0}
\begin{pmatrix} I & -I + A X(x;\delta) \tau_w J_0 \\ 
  \tau_u G_0 & \tau_w J_0 - B X(x;\delta) \end{pmatrix}
\begin{pmatrix} U \\ W \end{pmatrix}
= \begin{pmatrix} \xi \\ -\tau_u \nu \end{pmatrix}.
\ee

\begin{Rk}
Importantly, in the flat--interface case, $g(x) \equiv 0$, the
DNOs can be explicitly specified in terms of Fourier multipliers.
Considering the upper layer DNO, $G_0$, we recall the Rayleigh expansions
\cite{Petit80,Yeh05}
\be
\label{Eqn:Ray:u}
u(x,z) = \sump \hat{U}_p e^{i \alpha_p x + i \gamma_{u,p} z},
\ee
where
\bse
\label{Eqn:AlphaGamma}
\be
\alpha_p = \alpha + (2 \pi/d) p,
\quad
\gamma_{m,p} = \begin{cases}
  \sqrt{k_m^2 - \alpha_p^2}, & p \in \cU_m, \\
  i \sqrt{\alpha_p^2 - k_m^2 }, & p \not \in \cU_m,
  \end{cases}
\quad
m \in \{ u, w \},
\ee
and the propagating modes are
\be
\cU_m := \left\{ p \in \Integer\ |\ \alpha_p^2 \leq k_m^2 \right\},
\quad
m \in \{ u, w \},
\ee
\ese
which gives the exact solution of \eqref{Eqn:Helm:u} with Dirichlet data
$u(x,0) = U(x)$. From this the Neumann data can readily be shown to be
\bes
\tilde{U}(x) = -\pz u(x,0)
  = \sump -i \gamma_{u,p} \hat{U}_p e^{i \alpha_p x},
\ees
which gives
\bes
G_0[U] = \sump -i \gamma_{u,p} \hat{U}_p e^{i \alpha_p x}
  =: -i \gamma_{u,D} U,
\ees
defining the order--one Fourier multiplier, $\gamma_{u,D}$.
In analogous fashion, based on the Rayleigh expansion solution
of \eqref{Eqn:Helm:w},
\be
\label{Eqn:Ray:w}
w(x,z) = \sump \hat{W}_p e^{i \alpha_p x - i \gamma_{w,p} z},
\ee
one can demonstrate that
\bes
J_0[W] = \sump -i \gamma_{w,p} \hat{W}_p e^{i \alpha_p x}
  =: -i \gamma_{w,D} W.
\ees
\end{Rk}

%
%

\subsection{A High--Order Perturbation of Envelopes Method}
\label{Sec:HOPE}

As we shall see, \eqref{Eqn:DNO:0} is straightforward to solve
provided that $X(x) \equiv X_0 \in \Real$. In this case the 
equations are diagonalized by the Fourier transform and the solution 
can be found wavenumber--by--wavenumber. We build upon this observation
by considering envelope functions of the form \eqref{Eqn:X}
and proceeding with (regular) perturbation theory. As we are considering
deformations of the envelope (through the parameter $\delta$), we term 
such a scheme
a ``High--Order Perturbation of Envelopes'' (HOPE)
method to contrast with ``High--Order Perturbation of Surfaces'' (HOPS)
algorithms where the height/slope of the interface \textit{shape} is 
the perturbation parameter \cite{NichollsReitich99}.

For this HOPE approach we posit expansions
\be
\label{Eqn:UW:Exp}
\{ U, W \} = \{ U, W \}(x;\delta)
  = \sumell \{ U_{\ell}, W_{\ell} \}(x) \delta^{\ell},
\ee
and derive recursive formulas for the $\{ U_{\ell}, W_{\ell} \}$. It is not
difficult to see that, at order $\ell \geq 0$, one must solve
\be
\label{Eqn:MVR:0}
\begin{pmatrix} I & -I + A X_0 \tau_w J_0 \\ 
  \tau_u G_0 & \tau_w J_0 - B X_0 \end{pmatrix}
\begin{pmatrix} U_{\ell} \\ W_{\ell} \end{pmatrix}
= \delta_{\ell,0} \begin{pmatrix} \xi \\ -\tau_u \nu \end{pmatrix}
  + \begin{pmatrix} -A X_1(x) \tau_w J_0 W_{\ell-1} \\ 
    B X_1(x) W_{\ell-1} \end{pmatrix},
\ee
where $\delta_{\ell,q}$ is the Kronecker delta, and $W_{-1} \equiv 0$.
We will presently show that \eqref{Eqn:UW:Exp} converge strongly in 
appropriate Sobolev spaces. Importantly, these recursions also result
in a numerical algorithm that delivers HOS accuracy.

\begin{Rk}
\label{Rk:NonRes}
As we have pointed out that the operators $G_0$ and $J_0$ are
diagonalized by the Fourier transform, we can state the condition 
of ``non--resonance'' which we require for uniqueness of solutions.
As we shall see, in Transverse Electric (TE) polarization 
($A = 0$ and $\tau_m = 1$) we will require that the determinant function
\begin{align}
\DeltaTE_p 
  & := \widehat{(G_0)}_p + \widehat{(J_0)}_p - B X_0 \notag \\
  & = -i \gamma_{u,p} - i \gamma_{w,p} 
  - i k_0 \hsigmaloc X_0 - i k_0 \hsigmanloc X_0 \alpha_p^2,
\label{Eqn:DeltaTE}
\end{align}
satisfies, for some $\mu > 0$, 
$\min_{-\infty < p < \infty} \left\{ \Abs{\DeltaTE_p} \right\} > \mu$.
In Transverse Magnetic (TM) polarization ($B = 0$) it must be that the 
determinant function
\begin{align}
\DeltaTM_p 
  & := \tau_u \widehat{(G_0)}_p + \tau_w \widehat{(J_0)}_p
    - \tau_u \tau_w A X_0 \widehat{(G_0)}_p \widehat{(J_0)}_p \notag \\
  & = -\tau_u i \gamma_{u,p} - \tau_w i \gamma_{w,p}
    + \tau_u \tau_w \left( \frac{\hsigmaloc}{i k_0} \right) 
    X_0 \gamma_{u,p} \gamma_{w,p} \notag \\
  & \quad + \tau_u \tau_w \left( \frac{\hsigmanloc}{i k_0} \right) 
    X_0 \gamma_{u,p} \gamma_{w,p} \alpha_p^2,
\label{Eqn:DeltaTM}
\end{align}
satisfies, for some $\mu > 0$,
$\min_{-\infty < p < \infty} \left\{ \Abs{\DeltaTM_p} \right\} > \mu$.
\end{Rk}

We now describe precise progress on this in the following lemma.
\begin{Lemma}
If $X_0 > 0$, for any $p \in \Integer$
\bes
\DeltaTE_p \neq 0, \quad \DeltaTM_p \neq 0.
\ees
\end{Lemma}
\begin{proof}
We begin with the notation
\bes
z = z' + i z'' \in \Complex,
\quad
\quad z', z'' \in \Real,
\ees
and recall that, for $\hsigmaloc = \hsigmaloc' + i \hsigmaloc''$,
\bes
\hsigmaloc' = \frac{(2 E_F e^2) \Gamma}
  {\epsilon_0 c_0 (\Gamma^2 + h^2 f^2)} > 0,
\quad
\hsigmaloc'' = \frac{(2 E_F e^2) h f}
  {\epsilon_0 c_0 (\Gamma^2 + h^2 f^2)} > 0.
\ees
Furthermore, we have, c.f. \eqref{Eqn:AlphaGamma},
\bes
\gamma_{m,p} = \begin{cases} 
  \gamma_{m,p}', & p \in \cU_m, \\
  i \gamma_{m,p}'', & p \not \in \cU_m,
  \end{cases}
  \quad
  \gamma_{m,p}', \gamma_{m,p}'' \geq 0,
\ees
for $m \in \{ u, w \}$, so that either
\bes
\{ \gamma_{m,p}' \geq 0\ \text{and}\ \gamma_{m,p}'' = 0 \}
\quad \text{or} \quad
\{ \gamma_{m,p}' = 0\ \text{and}\ \gamma_{m,p}'' \geq 0 \}.
\ees
From the nonlocal current model we examine the term
\bes
Q = v_F^2 \left( \frac{3 f + 2 i/\tau}{4 f (f + i/\tau)^2} \right),
\ees
and calculate
\begin{align*}
Q & = \frac{v_F^2\tau(3f\tau+2i)}{4f(f\tau+i)^2}
    = \frac{v_F^2\tau}{4f(f^2\tau^2+1)^2}(3f\tau+2i)(f\tau-i)^2 \\
  & = \frac{v_F^2\tau}{4f(f^2\tau^2+1)^2}[f\tau(3f^2\tau^2+1)
     -2i(2f^2\tau^2+1)],
\end{align*}
therefore we have
\bes
Q' = \frac{v_F^2\tau^2(3f^2\tau^2+1)}{4(f^2\tau^2+1)^2} > 0,
\quad
Q'' = -\frac{v_F^2\tau(2f^2\tau^2+1)}{2f(f^2\tau^2+1)^2} < 0.
\ees
Finally, another crucial term in both polarizations is
\begin{align*}
\Sigma 
  & := \hsigmaloc \left( 1 + Q \alpha_p^2 \right) \\
  & = \left( \hsigmaloc' + i \hsigmaloc'' \right)
  \left( (1 + Q' \alpha_p^2) + i Q'' \alpha_p^2 \right) \\
  & = \left( \hsigmaloc' (1 + Q' \alpha_p^2)
  - \hsigmaloc'' Q'' \alpha_p^2 \right)
  + i \left( \hsigmaloc' Q'' \alpha_p^2 
  + \hsigmaloc'' (1 + Q' \alpha_p^2) \right).
\end{align*}
While the imaginary part is indeterminate, we can state that,
due to the signs of $\{ \hsigmaloc', \hsigmaloc'', Q', Q'' \}$,
the real part satisfies
\bes
\Sigma' = \RealPart{ \Sigma }
  = \hsigmaloc' (1 + Q' \alpha^2) - \hsigmaloc'' Q'' \alpha_p^2 > 0.
\ees

We begin with the case of TE polarization where we have,
from the parity of $\Sigma'$,
\bes
\RealPart{i \DeltaTE_p} = \gamma_{u,p}' + \gamma_{w,p}'
  + k_0 X_0 \Sigma' > 0.
\ees
For TM polarization we have
\bes
i \DeltaTM_p = \tau_u \gamma_{u,p} + \tau_w \gamma_{w,p}
  + \tau_u \gamma_{u,p} \tau_w \gamma_{w,p}
  \left( \frac{X_0}{k_0} \right) \Sigma,
\ees
and begin with the case of a Rayleigh singularity in the
upper layer, $\gamma_{u,p} = 0$. We point out that, if
$\epsilon_w \neq \epsilon_u$ then 
$\gamma_{w,p} \neq \gamma_{u,p} = 0$
for this choice of $p$. In this case
\bes
i \DeltaTM_p = \tau_w \gamma_{w,p} \neq 0.
\ees
Clearly, a Rayleigh singularity in the lower layer, 
$\gamma_{w,p} = 0$, can be handled simlarly. So, we now
fix on the situation of no Rayleigh singularities,
\bes
\{ \gamma_{m,p}' > 0\ \text{and}\ \gamma_{m,p}'' = 0 \}
\quad \text{or} \quad
\{ \gamma_{m,p}' = 0\ \text{and}\ \gamma_{m,p}'' > 0 \},
\ees
and divide the calculation into four parts:
\begin{enumerate}
\item Case $\gamma_{u,p}' = \gamma_{w,p}' = 0$. Here
  \bes
  \RealPart{i \DeltaTM_p} = -\tau_u \gamma_{u,p}'' \tau_w \gamma_{w,p}''
  \left( \frac{X_0}{k_0} \right) \Sigma' < 0.
  \ees
\item Case $\gamma_{u,p}' = \gamma_{w,p}'' = 0$. Here
  \bes
  \ImagPart{i \DeltaTM_p} = \tau_u \gamma_{u,p}''
    + \tau_u \gamma_{u,p}'' \tau_w \gamma_{w,p}' 
    \left( \frac{X_0}{k_0} \right) \Sigma'
    > 0.
  \ees
\item Case $\gamma_{u,p}'' = \gamma_{w,p}' = 0$. Here
  \bes
  \ImagPart{i \DeltaTM_p} = \tau_w \gamma_{w,p}''
    + \tau_u \gamma_{u,p}' \tau_w \gamma_{w,p}'' 
    \left(\frac{X_0}{k_0} \right) \Sigma'
    > 0.
  \ees
\item Case $\gamma_{u,p}'' = \gamma_{w,p}'' = 0$. Here
  \bes
  \RealPart{i \DeltaTM_p} = \tau_u \gamma_{u,p}' + \tau_w \gamma_{w,p}'
    + \tau_u \gamma_{u,p}' \tau_w \gamma_{w,p}'
    \left( \frac{X_0}{k_0} \right) \Sigma'
    > 0.
  \ees
\end{enumerate}
The conclusion of these four computations is that 
$\DeltaTM_p \neq 0$.
\end{proof}

%
%

\section{Analyticity}
\label{Sec:Anal}

Before describing our theoretical results we pause to specify
the function spaces we will require. For any real $s \geq 0$ we
recall the classical $L^2$--based Sobolev norm
\bes
\SobNorm{U}{s}^2 := \sump \Angle{p}^{2 s} \Abs{\hat{U}_p}^2,
\quad
\Angle{p}^2 := 1 + \Abs{p}^2,
\quad
\hat{U}_p := \frac{1}{d} \int_0^d U(x) e^{i \alpha_p x} \; dx,
\ees
which gives rise to the Sobolev space
\bes
H^s([0,d]) := \left\{ U(x) \in L^2([0,d])\ |\ 
  \SobNorm{U}{s} < \infty \right\}.
\ees
With this definition it is a simple matter to prove the following
Lemma.
\begin{Lemma}
\label{Lemma:G0J0}
For any real $s \geq 0$ there exist constants $C_G, C_J > 0$ such that
\bes
\SobNorm{G_0 U}{s} \leq C_G \SobNorm{U}{s+1},
\quad
\SobNorm{J_0 W}{s} \leq C_J \SobNorm{W}{s+1},
\ees
for any $U, W \in H^{s+1}$.
\end{Lemma}

We also recall, for any integer $s \geq 0$, the space of
$s$--times continuously differentiable functions with 
the H\"older norm
\bes
\HolderNorm{f}{s} = \max_{0 \leq \ell \leq s} \SupNorm{\px^{\ell} f}.
\ees
For later reference we recall the classical result
\cite{Evans10}.
\begin{Lemma}
\label{Lemma:Alg}
For any integer $s \geq 0$ there exists a constant $K = K(s)$ 
such that
\bes
\SobNorm{f U}{s} \leq K \HolderNorm{f}{s} \SobNorm{U}{s}.
\ees
\end{Lemma}

%
%


We now begin the rigorous analysis of the expansions \eqref{Eqn:UW:Exp}
and, for this, we appeal to the general theory of analyticity of
solutions of linear systems of equations. For a particular description
of the procedure, we follow the developments found in \cite{Nicholls16b}
for the solution of
\be
\label{Eqn:AVR}
\bA(\delta) \bV(\delta) = \bR(\delta),
\ee
which is (3.1) of \cite{Nicholls16b} with $\Eps$ replaced by $\delta$.
In \cite{Nicholls16b}, given expansions
\be
\label{Eqn:AR:Exp}
\bA(\delta) = \sumell \bA_{\ell} \delta^{\ell},
\quad
\bR(\delta) = \sumell \bR_{\ell} \delta^{\ell},
\ee
we seek a solution of the form
\be
\label{Eqn:V:Exp}
\bV(\delta) = \sumell \bV_{\ell} \delta^{\ell},
\ee
which satisfies
\bes
\bV_{\ell} = \bA_0^{-1} \left[ \bR_{\ell} 
  - \sum_{q=0}^{\ell-1} \bA_{\ell-q} \bV_q \right],
\quad
\ell \geq 0.
\ees
We restate the main result here for completeness.
\begin{Thm}[\cite{Nicholls16b}]
\label{Thm:Anal}
Given two Banach spaces $Y$ and $Z$, suppose that:
\begin{enumerate}[label=(H\arabic*)]
\item \label{Thm:Anal:H1}
  $\bR_{\ell} \in Z$ for all $\ell \geq 0$, and there
  exist constants $C_R > 0$, $B_R > 0$ such that
  \bes
  \Norm{\bR_{\ell}}{Y} \leq C_R B_R^{\ell},
  \quad
  \ell \geq 0.
  \ees
\item \label{Thm:Anal:H2}
  $\bA_{\ell}: Y \rightarrow Z$ for all $\ell \geq 0$, and
  there exists constants $C_A > 0$, $B_A > 0$ such that
  \bes
  \Norm{\bA_{\ell}}{Y \rightarrow Z} \leq C_A B_A^{\ell},
  \quad
  \ell \geq 0.
  \ees
\item \label{Thm:Anal:H3}
  $\bA_0^{-1}: Z \rightarrow Y$, and there exists a constant
  $C_e > 0$ such that
  \bes
  \Norm{\bA_0^{-1}}{Z \rightarrow Y} \leq C_e.
  \ees
\end{enumerate}
Then the equation \eqref{Eqn:AVR} has a unique solution \eqref{Eqn:V:Exp},
and there exist constants $C_V > 0$ and $B_V > 0$ such that
\bes
\Norm{\bV_{\ell}}{Y} \leq C_V B_V^{\ell},
\quad
\ell \geq 0,
\ees
for any
\bes
C_V \geq 2 C_e C_R,
\quad
B_V \geq \max \left\{ B_R, 2 B_A, 4 C_e C_A B_A \right\},
\ees
which implies that, for any $0 \leq \rho < 1$, \eqref{Eqn:V:Exp}
converges for all $\delta$ such that $B_V \delta < \rho$, i.e.,
$\delta < \rho/B_V$.
\end{Thm}

From \eqref{Eqn:DNO:0} it is easy to identify
\bes
\bA = \begin{pmatrix} I & -I + A X(x;\delta) \tau_w J_0 \\ 
  \tau_u G_0 & \tau_w J_0 - B X(x;\delta) \end{pmatrix},
\quad
\bV = \begin{pmatrix} U \\ W \end{pmatrix},
\quad 
\bR = \begin{pmatrix} \xi \\ -\tau_u \nu \end{pmatrix}.
\ees
All that remains is to find the forms \eqref{Eqn:AR:Exp}, and establish
Hypotheses~\ref{Thm:Anal:H1}, \ref{Thm:Anal:H2}, and 
\ref{Thm:Anal:H3}. As we shall shortly see, the analysis depends 
strongly upon the polarization (TE/TM) of our fields so we break our 
developments into these two cases. 

%
%

\subsection{Transverse Electric Polarization}
\label{Sec:Anal:TE}

In Transverse Electric polarization $A \equiv 0$ and $\tau_m = 1$, and we 
see that \eqref{Eqn:DNO:0} becomes
\be
\label{Eqn:TE}
\begin{pmatrix} I & -I \\ 
  G_0 & J_0 - B X(x;\delta) \end{pmatrix}
\begin{pmatrix} U \\ W \end{pmatrix}
= \begin{pmatrix} \xi \\ -\nu \end{pmatrix},
\ee
so that
\bes
\bA_0 = \begin{pmatrix} I & -I \\ 
  G_0 & J_0 - B X_0 \end{pmatrix};
\quad
\bA_1 = \begin{pmatrix} 0 & 0 \\ 
  0 & -B X_1(x) \end{pmatrix};
\quad
\bA_{\ell} \equiv \begin{pmatrix} 0 & 0 \\ 0 & 0 \end{pmatrix},
\quad
\ell \geq 2,
\ees
and
\bes
\bR_0 = \begin{pmatrix} \xi \\ -\nu \end{pmatrix};
\quad
\bR_{\ell} \equiv \begin{pmatrix} 0 \\ 0 \end{pmatrix},
\quad
\ell \geq 1.
\ees
As we shall see in the next Lemma, the natural spaces in which to
work for TE polarization are, for real $s \geq 0$,
\bes
Y = H^{s+1} \times H^{s+2},
\quad
Z = H^{s+1} \times H^s,
\ees
so that
\bes
\Norm{\underline{y}}{Y}^2
  = \SobNorm{\underline{y}_1}{s+1}^2
  + \SobNorm{\underline{y}_2}{s+2}^2,
\quad
\Norm{\underline{z}}{Z}^2
  = \SobNorm{\underline{z}_1}{s+1}^2
  + \SobNorm{\underline{z}_2}{s}^2.
\ees

\noindent \textbf{Hypothesis~\ref{Thm:Anal:H1}:}
With these definitions it is a simple matter to show that
\bes
\Norm{\bR_0}{Z}^2 = \SobNorm{\xi}{s+1}^2 
  + \SobNorm{\nu}{s}^2 < \infty,
\ees
given that
\bes
\xi = -e^{i \alpha x},
\quad
\nu = i \gamma_u e^{i \alpha x},
\ees
so that $\xi, \nu \in H^t$ for any real $t \geq 0$. Thus
Hypothesis~\ref{Thm:Anal:H1} is established with any choices of 
$C_R$ and $B_R$ such that $C_R B_R = \Norm{\bR_0}{Z}$.

\noindent \textbf{Hypothesis~\ref{Thm:Anal:H2}:} Considering generic
$U \in H^{s+1}$ and $W \in H^{s+2}$ we study
\begin{align*}
\Norm{\bA_0 \begin{pmatrix} U \\ W \end{pmatrix}}{Z}^2
  & = \SobNorm{U - W}{s+1}^2 
    + \SobNorm{G_0 U + J_0 W - B X_0 W}{s}^2 \\
  & \leq \SobNorm{U}{s+1}^2 + \SobNorm{W}{s+1}^2
    + C_G^2 \SobNorm{U}{s+1}^2
    + C_J^2 \SobNorm{W}{s+1}^2
\\ & \quad 
    + \Abs{i k_0 \hsigmaloc}^2 \Abs{X_0}^2 \SobNorm{W}{s}^2
    + \Abs{i k_0 \hsigmanloc}^2 \Abs{X_0}^2 \SobNorm{\px^2 W}{s}^2 \\
  & \leq C_0^2 \left( \SobNorm{U}{s+1}^2 
  + \SobNorm{W}{s+2}^2 \right) \\
  & = C_0^2 \Norm{\begin{pmatrix} U \\ W \end{pmatrix}}{Y}^2,
\end{align*}
where we have used Lemma~\ref{Lemma:G0J0},
and we have the desired mapping property of $\bA_0$. We turn to 
$\bA_1$ and find
\begin{align*}
\Norm{\bA_1 \begin{pmatrix} U \\ W \end{pmatrix}}{Z}^2
  & = \SobNorm{-B X_1(x) W}{s}^2 \\
  & \leq \Abs{i k_0 \hsigmaloc}^2 K^2 \HolderNorm{X_1}{s}^2 
    \SobNorm{W}{s}^2
    + \Abs{i k_0 \hsigmanloc}^2 K^2 \HolderNorm{X_1}{s}^2 
    \SobNorm{\px^2 W}{s}^2 \\
  & \leq C_1^2 \HolderNorm{X_1}{s}^2 \SobNorm{W}{s+2}^2 \\
  & \leq C_1^2 \HolderNorm{X_1}{s}^2
  \Norm{\begin{pmatrix} U \\ W \end{pmatrix}}{Y}^2,
\end{align*}
where we have used the Algebra property, Lemma~\ref{Lemma:Alg}, which
mandates integer $s \geq 0$. Thus, we are done with 
Hypothesis~\ref{Thm:Anal:H2} 
if we choose $C_A = \max \{ C_0, C_1 \}$ and 
$B_A = \HolderNorm{X_1}{s}$.

\noindent \textbf{Hypothesis~\ref{Thm:Anal:H3}:} The crux of the matter,
as always in regular perturbation theory, is the invertibility
of the linearized operator $\bA_0$ and its mapping properties.
For this we prove the following result.
\begin{Lemma}
\label{Lemma:TE}
Given real $s \geq 0$ if $Q \in H^{s+1}$ and $R \in H^s$, and
$X_0 \neq 0$ then there exists a unique solution of
\be
\label{Eqn:TE:Lin}
\begin{pmatrix} I & -I \\ G_0 & J_0 - B X_0 \end{pmatrix}
\begin{pmatrix} U \\ W \end{pmatrix}
= \begin{pmatrix} Q \\ R \end{pmatrix},
\ee
satisfying
\begin{align*}
\SobNorm{U}{s+1} 
  & \leq C_e \left\{ \SobNorm{Q}{s+1} + \SobNorm{R}{s} \right\}, \\
\SobNorm{W}{s+2} 
  & \leq C_e \left\{ \SobNorm{Q}{s+1} + \SobNorm{R}{s} \right\},
\end{align*}
for some constant $C_e > 0$.
\end{Lemma}
\begin{proof}
Upon expressing
\bes
U(x) = \sump \hat{U}_p e^{i \alpha_p x},
\quad
W(x) = \sump \hat{W}_p e^{i \alpha_p x},
\ees
we find that \eqref{Eqn:TE:Lin} demands
\bes
\begin{pmatrix} 1 & -1 \\ 
  -i \gamma_{u,p} & -i \gamma_{w,p} - \Bloc X_0 + \Bnloc X_0 \end{pmatrix}
\begin{pmatrix} \hat{U}_p \\ \hat{W}_p \end{pmatrix}
= \begin{pmatrix} \hat{Q}_p \\ \hat{R}_p \end{pmatrix}.
\ees
The exact solution is easily seen to be
\begin{align*}
\hat{U}_p 
  & = \frac{\{ i \gamma_{w,p} + (i k_0) \hsigmaloc X_0 
  + (i k_0) \hsigmanloc X_0 \alpha_p^2 \} \hat{Q}_p 
  + \hat{R}_p}{\DeltaTE_p},
\\
\hat{W}_p 
  & = \frac{i \gamma_{u,p} \hat{Q}_p + \hat{R}_p}{\DeltaTE_p}.
\end{align*}
Since we are ``nonresonant'' (see Remark~\ref{Rk:NonRes}) 
and, since $X_0 \neq 0$,
$1/\DeltaTE_p = \BigOh{\Angle{p}^{-2}}$ as $p \rightarrow \infty$
we find
\bes
\SobNorm{U}{s+1}^2 
  = \sump \Angle{p}^{2(s+1)} \Abs{\hat{U}_p}^2
  \leq \sump \Angle{p}^{2(s+1)} \left\{ C_Q \Abs{\hat{Q}_p}^2
    + C_R \Angle{p}^{-4} \Abs{\hat{R}_p}^2 \right\},
\ees
which delivers
\bes
\SobNorm{U}{s+1} \leq C_e \left\{ \SobNorm{Q}{s+1} + \SobNorm{R}{s-1}
  \right\}
  \leq C_e \left\{ \SobNorm{Q}{s+1} + \SobNorm{R}{s}
  \right\}.
\ees
In a similar manner, 
\bes
\SobNorm{W}{s+2}^2 
  = \sump \Angle{p}^{2(s+2)} \Abs{\hat{W}_p}^2
  \leq \sump \Angle{p}^{2(s+2)} \left\{ \Angle{p}^{-2} C_Q \Abs{\hat{Q}_p}^2
    + C_R \Angle{p}^{-4} \Abs{\hat{R}_p}^2 \right\},
\ees
which gives
\bes
\SobNorm{W}{s+2} \leq C_e \left\{ \SobNorm{Q}{s+1} + \SobNorm{R}{s}
  \right\}.
\ees
\end{proof}

Having established Hypotheses~\ref{Thm:Anal:H1}, \ref{Thm:Anal:H2},
and \ref{Thm:Anal:H3} we can invoke Theorem~\ref{Thm:Anal} to deduce.
\begin{Thm}
\label{Thm:Anal:TE}
Given an integer $s \geq 0$, if $X_0 \neq 0$ and $X_1 \in C^s([0,d])$
there exists
a unique solution pair, \eqref{Eqn:UW:Exp}, of the TE problem
\eqref{Eqn:TE} satisfying
\be
\label{Eqn:Anal:TE}
\SobNorm{U_{\ell}}{s+1} \leq C_U D^{\ell},
\quad
\SobNorm{W_{\ell}}{s+2} \leq C_W D^{\ell},
\quad
\forall\ \ell \geq 0,
\ee
for any $D > C \HolderNorm{X_1}{s}$ where $C_U$ and $C_W$ are
constants.
\end{Thm}

%
%

\subsection{Transverse Magnetic Polarization}
\label{Sec:Anal:TM}

Meanwhile, in Transverse Magnetic polarization $B \equiv 0$ and
we see that \eqref{Eqn:DNO:0} becomes
\be
\label{Eqn:TM}
\begin{pmatrix} I & -I + A X \tau_w J_0\\ 
  \tau_u G_0 & \tau_w J_0 \end{pmatrix}
\begin{pmatrix} U \\ W \end{pmatrix}
= \begin{pmatrix} \xi \\ -\tau_u \nu \end{pmatrix},
\ee
so that
\bes
\bA_0 = \begin{pmatrix} I & -I + A X_0 \tau_w J_0 \\ 
  \tau_u G_0 & \tau_w J_0 \end{pmatrix};
\quad
\bA_1 = \begin{pmatrix} 0 & A X_1(x) \tau_w J_0 \\ 
  0 & 0 \end{pmatrix};
\quad
\bA_{\ell} \equiv \begin{pmatrix} 0 & 0 \\ 0 & 0 \end{pmatrix},
\quad
\ell \geq 2,
\ees
and,
\bes
\bR_0 = \begin{pmatrix} \xi \\ -\tau_u \nu \end{pmatrix};
\quad
\bR_{\ell} \equiv \begin{pmatrix} 0 \\ 0 \end{pmatrix},
\quad
\ell \geq 1.
\ees
It will become clear presently that the natural spaces for TM
polarization are, for real $s \geq 0$,
\bes
Y = H^{s+1} \times H^{s+3},
\quad
Z = H^s \times H^s,
\ees
so that
\bes
\Norm{\underline{y}}{Y}^2
  = \SobNorm{\underline{y}_1}{s+1}^2
  + \SobNorm{\underline{y}_2}{s+3}^2,
\quad
\Norm{\underline{z}}{Z}^2
  = \SobNorm{\underline{z}_1}{s}^2
  + \SobNorm{\underline{z}_2}{s}^2.
\ees

\noindent \textbf{Hypothesis~\ref{Thm:Anal:H1}:}
Akin to the TE case
\bes
\Norm{\bR_0}{Z}^2 = \SobNorm{\xi}{s}^2 
  + \SobNorm{\tau_u \nu}{s}^2 < \infty,
\ees
and Hypothesis~\ref{Thm:Anal:H1} is established with any choices of 
$C_R$ and $B_R$ such that $C_R B_R = \Norm{\bR_0}{Z}$.

\noindent \textbf{Hypothesis~\ref{Thm:Anal:H2}:} Once again, 
considering generic $U \in H^{s+1}$ and $W \in H^{s+3}$ we consider
\begin{align*}
\Norm{\bA_0 \begin{pmatrix} U \\ W \end{pmatrix}}{Z}^2
  & = \SobNorm{U - W + A X_0 \tau_w J_0 W}{s}^2 
    + \SobNorm{\tau_u G_0 U + \tau_w J_0 W}{s}^2 \\
  & \leq \SobNorm{U}{s}^2 + \SobNorm{W}{s}^2
    + \Abs{\tau_u}^2 C_G^2 \SobNorm{U}{s+1}^2 
\\ & \quad 
    + \left\{ \Abs{\frac{\hsigmaloc}{i k_0}}^2 \Abs{X_0}^2 + 1 \right\}
      \Abs{\tau_w}^2 C_J^2 \SobNorm{W}{s+1}^2
    + \left\{ \Abs{\frac{\hsigmanloc}{i k_0}}^2 \Abs{X_0}^2 \right\}
      \Abs{\tau_w}^2 C_J^2 \SobNorm{\px^2 W}{s+1}^2\\
  & \leq C_0^2 \left( \SobNorm{U}{s+1}^2 
  + \SobNorm{W}{s+3}^2 \right) \\
  & = C_0^2 \Norm{\begin{pmatrix} U \\ W \end{pmatrix}}{Y}^2,
\end{align*}
again using Lemma~\ref{Lemma:G0J0}, and we have the required mapping
property of $\bA_0$. We now consider $\bA_1$
\begin{align*}
\Norm{\bA_1 \begin{pmatrix} U \\ W \end{pmatrix}}{Z}^2
  & = \SobNorm{A X_1(x) \tau_w J_0 W}{s}^2 \\
  & \leq \Abs{\frac{\hsigmaloc}{i k_0}}^2 K^2 \HolderNorm{X_1}{s}^2 
    \Abs{\tau_w}^2 \SobNorm{W}{s+1}^2
    + \Abs{\frac{\hsigmanloc}{i k_0}}^2 K^2 \HolderNorm{X_1}{s}^2 
    \Abs{\tau_w}^2 \SobNorm{\px^2 W}{s+1}^2 \\
  & \leq C_1^2 \HolderNorm{X_1}{s}^2 \SobNorm{W}{s+3}^2 \\
  & \leq C_1^2 \HolderNorm{X_1}{s}^2
  \Norm{\begin{pmatrix} U \\ W \end{pmatrix}}{Y}^2,
\end{align*}
where we have used Lemma~\ref{Lemma:Alg}. Thus, we are done with 
Hypothesis~\ref{Thm:Anal:H2} if we choose $C_A = \max \{ C_0, C_1 \}$
and $B_A = \HolderNorm{X_1}{s}$.

\noindent \textbf{Hypothesis~\ref{Thm:Anal:H3}:} We now study
the invertibility of the operator $\bA_0$.
\begin{Lemma}
\label{Lemma:TM}
Given real $s \geq 0$ if $Q \in H^s$, $R \in H^s$, and $X_0 \neq 0$
then there exists a unique solution of
\be
\label{Eqn:TM:Lin}
\begin{pmatrix} I & -I + A X_0 \tau_w J_0 \\ 
  \tau_u G_0 & \tau_w J_0 \end{pmatrix}
\begin{pmatrix} U \\ W \end{pmatrix}
= \begin{pmatrix} Q \\ R \end{pmatrix},
\ee
satisfying
\begin{align*}
\SobNorm{U}{s+1} 
  & \leq C_e \left\{ \SobNorm{Q}{s} + \SobNorm{R}{s} \right\}, \\
\SobNorm{W}{s+3} 
  & \leq C_e \left\{ \SobNorm{Q}{s} + \SobNorm{R}{s} \right\},
\end{align*}
for some constant $C_e > 0$.
\end{Lemma}
\begin{proof}
With
\bes
U(x) = \sump \hat{U}_p e^{i \alpha_p x},
\quad
W(x) = \sump \hat{W}_p e^{i \alpha_p x},
\ees
we find that \eqref{Eqn:TM:Lin} requires
\bes
\begin{pmatrix} 1 & -1 - A X_0 \tau_w i \gamma_{w,p} \\ 
  -\tau_u i \gamma_{u,p} & -\tau_w i \gamma_{w,p} \end{pmatrix}
\begin{pmatrix} \hat{U}_p \\ \hat{W}_p \end{pmatrix}
= \begin{pmatrix} \hat{Q}_p \\ \hat{R}_p \end{pmatrix}.
\ees
The exact solution is easily seen to be
\begin{align*}
\hat{U}_p 
  & = \frac{-\tau_w i \gamma_{w,p} \hat{Q}_p 
    + \left\{ 1 + \frac{X_0}{i k_0} (\hsigmaloc
    + \hsigmanloc \alpha_p^2) \right\} 
      \tau_w i \gamma_{w,p} \hat{R}_p}{\DeltaTM_p}, \\
\hat{W}_p 
  & = \frac{\tau_u i \gamma_{u,p} \hat{Q}_p + \hat{R}_p}{\DeltaTM_p}.
\end{align*}
Once again, as we are ``nonresonant'' (Remark~\ref{Rk:NonRes}) 
and, since $X_0 \neq 0$,
$1/\DeltaTM_p = \BigOh{\Angle{p}^{-4}}$ as $p \rightarrow \infty$
we find
\bes
\SobNorm{U}{s+1}^2 
  = \sump \Angle{p}^{2(s+1)} \Abs{\hat{U}_p}^2
  \leq \sump \Angle{p}^{2(s+1)} 
    \left\{ C_Q \Angle{p}^{-6} \Abs{\hat{Q}_p}^2
    + C_R \Angle{p}^{-2} \Abs{\hat{R}_p}^2 \right\},
\ees
which gives
\bes
\SobNorm{U}{s+1} \leq C_e \left\{ \SobNorm{Q}{s-2} + \SobNorm{R}{s}
  \right\}
  \leq C_e \left\{ \SobNorm{Q}{s} + \SobNorm{R}{s}
  \right\}.
\ees
Similarly,
\bes
\SobNorm{W}{s+3}^2 
  = \sump \Angle{p}^{2(s+3)} \Abs{\hat{W}_p}^2
  \leq \sump \Angle{p}^{2(s+3)} \left\{ C_Q \Angle{p}^{-6} \Abs{\hat{Q}_p}^2
    + C_R \Angle{p}^{-8} \Abs{\hat{R}_p}^2 \right\},
\ees
which delivers
\bes
\SobNorm{W}{s+3} \leq C_e \left\{ \SobNorm{Q}{s} + \SobNorm{R}{s-1}
  \right\}
  \leq C_e \left\{ \SobNorm{Q}{s} + \SobNorm{R}{s}
  \right\}.
\ees
\end{proof}

Having established Hypotheses~\ref{Thm:Anal:H1}, \ref{Thm:Anal:H2},
and \ref{Thm:Anal:H3} we can invoke Theorem~\ref{Thm:Anal} to deduce
the desired result.
\begin{Thm}
\label{Thm:Anal:TM}
Given an integer $s \geq 0$, if $X_0 \neq 0$ and $X_1 \in C^s([0,d])$
there exists
a unique solution pair, \eqref{Eqn:UW:Exp}, of the TM problem
\eqref{Eqn:TM} satisfying
\be
\label{Eqn:Anal:TM}
\SobNorm{U_{\ell}}{s+1} \leq C_U D^{\ell},
\quad
\SobNorm{W_{\ell}}{s+3} \leq C_W D^{\ell},
\quad
\forall\ \ell \geq 0,
\ee
for any $D > C \HolderNorm{X_1}{s}$ where $C_U$ and $C_W$ are
constants.
\end{Thm}

%
%

\section{Numerical Results}
\label{Sec:NumRes}

We now discuss how the recursions outlined above can be implemented
in a HOS scheme for simulating the surface
scattered fields $\{ U, W \}$. After describing the implementation
we use our algorithm to simulate absorbance spectra of TM polarized
plane waves incident upon a periodic grid of graphene ribbons as
described in \cite{GDBP16}.

%
%

\subsection{Implementation}
\label{Sec:Impl}

A numerical implementation of our recursions is rather straightforward.
To begin, we must truncate the HOPE expansions \eqref{Eqn:UW:Exp} after
a finite number, $L$, of Taylor orders
\bes
\{ U, W \} \approx \{ U^L, W^L \}
  := \sum_{\ell=0}^{L} \{ U_{\ell}, W_{\ell} \}(x) \delta^{\ell},
\ees
which satisfy, in either TE or TM polarization, \eqref{Eqn:MVR:0}
up to perturbation order $L$. For this, in consideration of the
quasiperiodic boundary conditions and our HOS
philosophy \cite{GottliebOrszag77,ShenTang06,ShenTangWang11},
we utilize the finite Fourier representations
\bes
\{ U_{\ell}, W_{\ell} \} \approx \{ U_{\ell}^{N_x}, W_{\ell}^{N_x} \}
  := \sum_{p=-N_x/2}^{N_x/2-1} \{ \hat{U}_{\ell,p}, \hat{W}_{\ell,p} \}
    e^{i \alpha_p x},
\quad
0 \leq \ell \leq L,
\ees
delivering
\be
\label{Eqn:UW:NNx}
\{ U, W \} \approx \{ U^{L,N_x}, W^{L,N_x} \}
  = \sum_{\ell=0}^{L} \sum_{p=-N_x/2}^{N_x/2-1}
    \{ \hat{U}_{\ell,p}, \hat{W}_{\ell,p} \} e^{i \alpha_p x},
\ee
and, with a collocation approach, we simply demand that \eqref{Eqn:MVR:0}
be true at the equally--spaced gridpoints $x_j = (d/N_x) j$, 
$0 \leq j \leq N_x-1$.

Due to the fact that the operators $\{ G_0, J_0 \}$ are Fourier
multipliers, they can be readily applied in Fourier space after
a Discrete Fourier Transform (DFT) which we accelerate by the 
Fast Fourier Transform (FFT) algorithm. Finally, we evaluate 
multiplication by the function
$X_1(x)$ on the physical side, pointwise at the equally--spaced
gridpoints $x_j$.

As with all perturbation schemes it is important to specify
how the Taylor series in \eqref{Eqn:UW:NNx} are to be summed.
On the one hand, ``direct'' Taylor summation seems natural,
however, this method is limited to the \textit{disk} of analyticity
centered at the origin. However, it has been our experience that
the actual domain of analyticity is much larger and may include
the entire real axis (despite poles on the imaginary axis and
elsewhere in the complex plane far from the real axis)
\cite{NichollsReitich00b}. One way
to access this extended region of analyticity is the classical
technique of Pad\'e approximation \cite{BakerGravesMorris96}
which has been used successfully for enhancing HOPS schemes
in the past \cite{NichollsReitich99,NichollsReitich00b,NichollsReitich03b}.
Pad\'e approximation seeks to estimate the truncated Taylor series
$f(\delta) = \sum_{\ell=0}^{L} f_{\ell} \delta^{\ell}$ 
by the rational function
\bes
\left[ \frac{M}{N} \right](\delta) := \frac{a^M(\delta)}{b^N(\delta)}
  = \frac{\sum_{m=0}^{M} a_m \delta^m}{\sum_{n=0}^{N} b_n \delta^n},
\quad
M + N = L,
\ees
and
\bes
\left[ \frac{M}{N} \right](\delta) = f(\delta) + \BigOh{\delta^{M+N+1}};
\ees
well--known formulas for the coefficients $\{ a_m, b_n \}$ can be
found in \cite{BakerGravesMorris96}. These Pad\'e approximants have
stunning properties of enhanced convergence, and we point the
interested reader to \S~2.2 of \cite{BakerGravesMorris96} and the
calculations in \S~8.3 of \cite{BenderOrszag78} for a complete
discussion.

%
%

\subsection{Absorbance Spectra}
\label{Sec:AbsSpec}

With an implementation of our algorithm we can now address questions
of importance to practitioners. As a specific example, we consider the
work of Goncalves, Dias, Bludov, and Peres \cite{GDBP16} who studied the
scattering of linear waves by arrays of graphene ribbons mounted
between dielectric layers. More specifically we refer the reader to
Figure~4 of \cite{GDBP16} which shows the results of their investigations
into the effect of the ribbon period on the frequency of a Graphene
Surface Plasmon (GSP) excited by the configuration.

To generate this figure \cite{GDBP16} focused upon TM
polarization, set the physical parameters
\be
\label{Eqn:GDBP}
\epsilon_u = 3,
\quad
\epsilon_w = 4,
\quad
E_F = 0.4~\text{eV},
\quad
\Gamma = 3.7~\text{meV},
\ee
and studied normal incidence so that $\theta=\alpha=0$. The lateral period 
(which they denoted $L$) of the structure was varied among 
$d = 1, 2, 4, 8$ (in microns) while the width of the graphene in each
period cell was set to $d/2$.

In the study of diffraction gratings, quantities of great physical interest
are the efficiencies. Recalling the Rayleigh expansions, \eqref{Eqn:Ray:u}
and \eqref{Eqn:Ray:w}, and the definitions, \eqref{Eqn:AlphaGamma}, these
are given by
\bes
e_{u,p} := \frac{\gamma_{u,p} \Abs{\Hat{U}_p}^2}{\gamma_{u,0}},
\quad
e_{w,p} := \frac{\gamma_{w,p} \Abs{\Hat{W}_p}^2}{\gamma_{u,0}}.
\ees
With these we can define the reflectance, transmittance, and absorbance
respectively as
\bes
R := \sum_{p \in \cU_u} e_{u,p},
\quad
T := \sum_{p \in \cU_w} e_{w,p},
\quad
A := 1 - R - \frac{\epsilon_u}{\epsilon_w} T;
\ees
we note that all--dielectric structures possess a principle of conservation
of energy which mandates $A=0$. However, as graphene has noteworthy metallic 
properties,
an indicator of a plasmonic response is given by a significant deviation
of $A$ from zero. Figure~4 of \cite{GDBP16} is a plot of precisely this
quantity, versus a range of illumination frequencies, for the four values
of $d$ mentioned above. In particular, we note significant peaks in
$A$, the ``absorbance spectra,'' of magnitude 0.35 in the vicinities
of $\nu = 2, 4, 6, 8$~THz for the values $d = 8, 4, 2, 1$ microns,
respectively. In subsequent work by Fallali, Low, Tamagnone, and 
Perruisseau--Carrier \cite{FLowTP15}, this study was extended
(with slightly different parameters) to include the nonlocal effects
produced by the model we describe above.

With an implementation of our new recursions we attempted to recreate
this work with both the local and nonlocal models. Our results, with 
the same physical parameters, \eqref{Eqn:GDBP}, supplemented with
\be
\label{Eqn:GDBP:NL}
v_F = 1~\text{$\mu$m/s},
\quad
\tau = 0.09~\text{s},
\ee
and numerical values $N_x = 128$ and $L = 16$, are displayed in 
Figure~\ref{Fig:Absorbance:HOPE} for $\delta=1$. It is noteworthy
that Pad\'e approximation was \textit{required} to achieve these results
as Taylor summation diverged. We point out the remarkable
\textit{qualitative} agreement between the local results (dashed curves)
and those of \cite{GDBP16}, which we take as evidence for the accuracy 
and utility of our approach. In addition,
we point out the \textit{shifted} solid curves generated by the nonlocal
model, in particular the blueshift of the peaks to higher frequencies
$\nu$ as also noted in \cite{FLowTP15}.
%
%
%
%
%
%
\begin{figure}[hbt]
\begin{center}
  \includegraphics[width=0.8\textwidth]{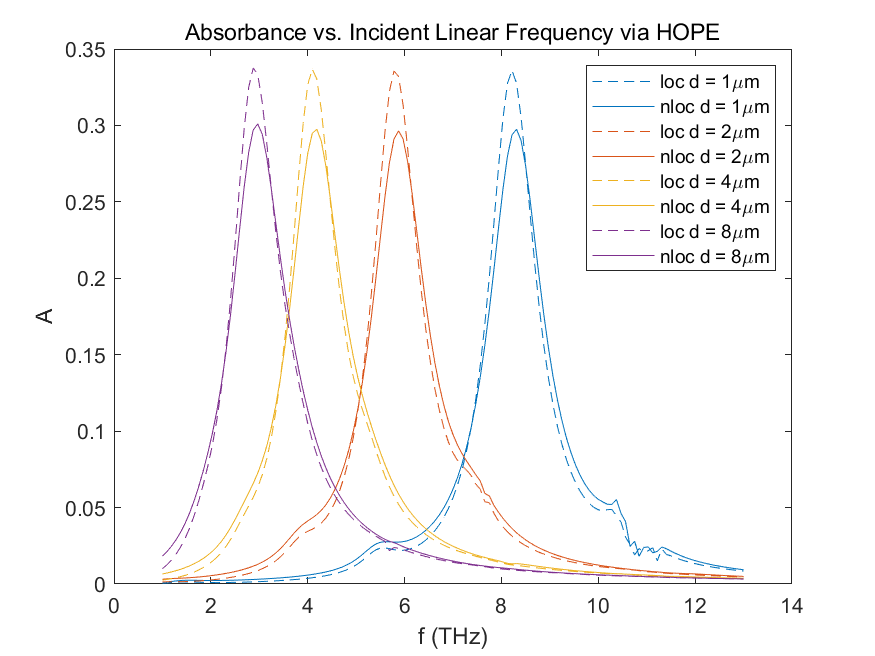}
  \caption{Plot of HOPE simulation of the absorbance spectra for 
    normally incident plane--wave illumination of a periodic array 
    of graphene ribbons with periodicity $d$ mounted between two 
    dielectrics. The physical parameters are specified in 
    \eqref{Eqn:GDBP} and \eqref{Eqn:GDBP:NL}, and the numerical
    parameters were $N_x=128$
    and $L=16$.}
  \label{Fig:Absorbance:HOPE}
\end{center}
\end{figure}

Of course it is always useful to have additional validation, and
for this we pondered the question of simply approximating the
governing equations \eqref{Eqn:DNO:0}
with $\delta=1$ using a collocation approach
\cite{GottliebOrszag77,ShenTang06,ShenTangWang11}: Expand the
\bes
\{ U, W \} \approx \{ U^{N_x}, W^{N_x} \}
  = \sum_{p=-N_x/2}^{N_x/2-1} \{ \hat{U}_p, \hat{W}_p \} e^{i \alpha_p x},
\ees
and demand that \eqref{Eqn:DNO:0} be true at the gridpoints
$x_j = (d/N_x) j$, $0 \leq j \leq N_x-1$. We implemented this algorithm
and achieved the results displayed in 
Figure~\ref{Fig:Absorbance:Collocation}.
Interestingly, the difference between these collocation results and
our HOPE computations is largely negligible. (The ``wiggles'' seen
in the high--frequency regime of Figure~\ref{Fig:Absorbance:HOPE}
we attribute to the somewhat unstable nature of Pad\'e summation
which is sometimes observed \cite{BenderOrszag78}.)
Importantly, with non--optimized MATLAB\texttrademark \cite{MATLAB}
implementations of each
algorithm, our new HOPE approach is nearly ten times faster than
the collocation approach. For this reason we find our new algorithm
to be quite compelling, though we intend to study this issue in
a variety of settings in a forthcoming publication.
%
%
%
%
%
%
\begin{figure}[hbt]
\begin{center}
  \includegraphics[width=0.8\textwidth]{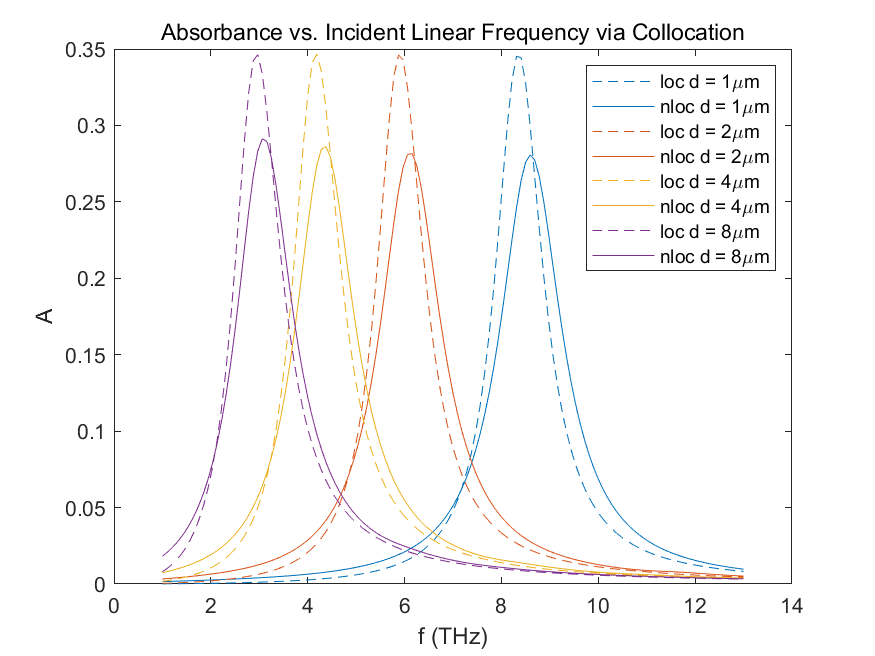}
  \caption{Plot of collocation simulation of the absorbance 
  spectra for 
    normally incident plane--wave illumination of a periodic array 
    of graphene ribbons with periodicity $d$ mounted between two 
    dielectrics. The physical parameters are specified in 
    \eqref{Eqn:GDBP} and \eqref{Eqn:GDBP:NL}, and the numerical
    parameter was $N_x=128$.}
  \label{Fig:Absorbance:Collocation}
\end{center}
\end{figure}
%

%
%

\section*{Acknowledgments}

D.P.N.\ gratefully acknowledges support from the 
National Science Foundation through grant No.~DMS--2111283.

%
%

\bibliography{nicholls}

\end{document}